\documentclass[11pt,twoside,a4paper]{amsart}
\usepackage{amssymb}
\date{\today}

\def\W{{\mathcal W}}

\def\depth{\text{depth}\,}

\def\w{\wedge}

\def\dbar{\bar\partial}

\def\C{{\mathbb C}}
\def\w{{\wedge}}

\def\A{{\mathcal A}}
\def\B{{\mathbb B}}

\def\supp{\text{supp}\,}

\def\S{{\mathcal S}}

\def\F{{\mathcal F}}
\def\Pr{{\mathcal P}}
\def\K{{\mathcal K}}

\def\codim{{\rm codim\,}}

\def\K{{\mathcal K}}
\def\Ker{{\rm Ker\,  }}
\def\rank{{\rm rank\, }}
\def\Dom{{\rm Dom\,  }}

\def\E{{\mathcal E}}

\def\Ok{{\mathcal O}}

\def\L{{\mathcal L}}

\def\Re{{\rm Re\,  }}

\def\L{{\mathcal L}}

\def\depth{{\rm depth\,}}
\def\J{{\mathcal J}}
\def\nbh{neighborhood }
\def\PM{{\mathcal{PM}}}
\def\HM{{\mathcal{PM}}}
\def\be{\begin{equation}}
\def\ee{\end{equation}}

\newtheorem{thm}{Theorem}[section]
\newtheorem{lma}[thm]{Lemma}
\newtheorem{cor}[thm]{Corollary}
\newtheorem{prop}[thm]{Proposition}

\theoremstyle{definition}

\theoremstyle{remark}

\newtheorem{preremark}{Remark}
\newtheorem{preex}{Example}

\newenvironment{remark}{\begin{preremark}}{\qed\end{preremark}}
\newenvironment{ex}{\begin{preex}}{\qed\end{preex}}

\numberwithin{equation}{section}

\title[]{Weighted Koppelman formulas and 
the $\dbar$-equation  on  an analytic space}

\begin{document}

\date{\today}

\author{Mats Andersson \& H\aa kan Samuelsson}

\address{M. Andersson, Department of Mathematics\\Chalmers University of Technology and the University of 
Gothenburg\\S-412 96 G\"OTEBORG\\SWEDEN}
\address{H. Samuelsson, Matematisk Institutt, Universitetet i Oslo, Postboks 1053 Blindern, 0316 Oslo, Norway}

\email{matsa@chalmers.se, haakansa@math.uio.no}


\thanks{The first author was
  partially supported by the Swedish 
  Research Council; the second author was partially supported by
  a Post Doctoral Fellowship from the Swedish 
  Research Council.}

\begin{abstract}
Let $X$ be an analytic space of pure dimension.
We introduce  a formalism  to generate intrinsic weighted Koppelman formulas
on $X$ that provide  solutions to the $\dbar$-equation. 
We obtain new existence results
for the  $\dbar$-equation, 
as well as new proofs of various known results.
\end{abstract}

\maketitle

\section{Introduction}

Let $X$ be an analytic space  of pure dimension $n$
and let $\Ok=\Ok^X$ be the structure sheaf of (strongly)
holomorphic functions.
Locally  $X$ is a subvariety of
a domain  $\Omega$ in $\C^N$ and then 
$\Ok^X=\Ok^\Omega/\J$, where $\J$ is the sheaf in $\Omega$ 
of holomorphic functions that vanish on  $X$.
In the same way we say that $\phi$ is  a  smooth $(0,q)$-form on $X$,
$\phi\in\E_{0,q}(X)$, if given a local embedding,  there is 
a  smooth form in a \nbh in the ambient space 
such that $\phi$ is its pull-back to $X_{reg}$.
It is well-known that this defines an intrinsic sheaf
$\E_{0,q}^X$ on $X$. 
It was proved in \cite{HPo} that if 
$X$ is embedded as a reduced complete intersection  in a pseudoconvex domain
and $\phi$ is a $\dbar$-closed smooth form on  $X$, then 
there is a solution $\psi$ to $\dbar\psi=\phi$ on $X_{reg}$.
It was  an open question for long  whether this
holds  more generally, and it was proved only in \cite{AS2}\footnote{The proof in
\cite{AS2} first appeared in \cite{AS1}.}
that this   is indeed true for  any Stein  space $X$.


In \cite{AS2} we  introduced   fine  (modules
over the sheaf of smooth forms) sheaves  $\A_k$ of $(0,k)$-currents on $X$,
which coincide with the sheaves of smooth forms on $X_{reg}$ and have rather
``mild'' singularities at $X_{sing}$. The main result in \cite{AS2} is that
\begin{equation}\label{resolution}
0\to \Ok^X\to \A_0\stackrel{\dbar}{\to}\A_1\stackrel{\dbar}{\to}
\end{equation}
is a (fine) resolution of $\Ok^X$. By the de~Rham theorem it follows that
the classical Dolbeault isomorphism for a smooth $X$ extends to an arbitrary
(reduced) singular space, but with the sheaves $\A_k$ instead of $\E_{0,k}$.
In particular, if $X$ is Stein, $\phi\in\A_{q+1}(X)$ and $\dbar\phi=0$, then
there is $u\in\A_q(X)$ such that $\dbar u=\phi$. 

\smallskip

The results in \cite{AS2} are based on semiglobal 
Koppelman formulas on $X$ that we first 
describe for smooth forms.

\begin{thm}\label{main} Let $X$ be an analytic subvariety of pure dimension $n$ of
a pseudoconvex domain $\Omega\subset\C^N$ and assume that $\Omega'\subset\subset \Omega$
and $X':=X\cap\Omega'$.
There are linear operators $\K\colon\E_{0,q+1}(X)\to\E_{0,q}(X_{reg}')$
and $\Pr\colon \E_{0,0}(X)\to\Ok(\Omega')$ such that
\begin{equation}\label{koppelman}
\phi(z)=\dbar \K\phi(z)+\K(\dbar\phi)(z),\quad z\in X'_{reg},\ 
\phi\in\E_{0,q}(X),\ q\ge 1,
\end{equation}
and
\begin{equation}\label{koppelman2}
\phi(z)= \K(\dbar\phi)(z)+ \Pr\phi(z),\quad z\in X'_{reg},\  \phi\in\E_{0,0}(X).
\end{equation}
Moreover, there is a number $M$ such that 
\begin{equation}\label{asymp}
\K\phi(z)=\Ok(\delta(z)^{-M}),
\end{equation}
where $\delta(z)$ is the distance to $X'_{sing}$.
\end{thm}

The operators are given as
\begin{equation}\label{kpv}
\K\phi(z)=\int_\zeta k(\zeta,z)\w \phi(\zeta),
\quad
\Pr\phi(z)=\int_\zeta p(\zeta,z)\w \phi(\zeta),
\end{equation}
where $k$ and $p$ are intrinsic integral kernels  
on $X\times X_{reg}'$ and $X\times \Omega'$,
respectively. They are locally integrable with respect
to $\zeta$ on $X_{reg}$ and the integrals in \eqref{kpv} 
are principal values at $X_{sing}$.
If $\phi$ vanishes in a neighborhood of a point $x$, then
$\K\phi$ is smooth at $x$. 
The  distance $\delta(z)$ is the one induced from the ambient space;
up to  a constant it is independent of the particular embedding.
The existence result in \cite{HPo} for a reduced complete intersection
is also obtained by an integral formula, which however 
does not give an intrinsic solution operator on $X$.

We cannot expect our solution $\K\phi$ to be smooth across $X_{sing}$,
see, e.g., Example~1  in \cite{AS2}.
However,  $\K$ and $\Pr$  extend to operators $\K\colon\A_{q+1}(X)\to \A_q(X')$
and $\Pr\colon\A_0(X)\to\Ok(\Omega')$,  and the Koppelman formulas still
hold, so in particular, $\dbar\K\phi=\phi$ if $\phi\in\A_{q+1}(X)$ and $\dbar\phi=0$
(Theorem~4 in \cite{AS2}).

\smallskip

There is an integer $L$, only depending on
$X$, such that for each $k\ge L$,
$\K\colon C^k_{0,q+1}(X)\to C^k_{0,q}(X_{reg}')$
and
$\Pr\colon C^k_{0,0}(X)\to\Ok(\Omega')$. Here $\phi\in C^k_{0,q}(X)$ means that $\phi$ is the 
pullback to $X_{reg}$ of a $(0,q)$-form of class $C^k$ in a \nbh of $X$ in the ambient space.
We have

\begin{thm}\label{maincor} Let $X,X',\Omega,\Omega'$ be as in the previous theorem.

\smallskip

\noindent (i)  \   If $\phi\in C^k_{0,q+1}(X)$, $q\ge 0$,
$k\ge L+1$,  and $\dbar\phi=0$, then there is 
$\psi\in C^k_{0,q}(X_{reg}')$ with
$\psi(z)=\Ok(\delta(z)^{-M})$ and $\dbar\psi=\phi$.

\smallskip

\noindent (ii)\  If $\phi\in\ C^{L+1}_{0,0}(X)$ and $\dbar\phi=0$ then
$\phi$ is strongly holomorphic.
\end{thm}

Part (ii)  is well-known, \cite{Ma} and \cite{Sp}, but 
$\Pr\phi$ provides an explicit holomorphic extension of $\phi$ to $\Omega'$.

\smallskip

Our solution operator $\K$ behaves like a classical solution operator
on $X_{reg}$ and by  introducing appropriate weight factors  in the integral operators
we get

\begin{thm}\label{main2}
Let $X,X',\Omega,\Omega'$ be as in the previous theorem.
Given $\mu\ge 0$ there is  $\mu'\ge 0$ and a linear operator 
$\K$ such that if $\phi$ is a $\dbar$-closed $(0,q+1)$-form on $X_{reg}$,
$q\ge 0$,  with
$\delta^{-\mu'}\phi \in L^p(X_{reg})$, $1\leq p\leq \infty$, then
$\dbar\K\phi=\phi$ and $\delta^{-\mu}\K \phi \in L^p(X_{reg}')$.
\end{thm}

The existence of such solutions was  proved  in \cite{FOV2} 
(even for $(r,q)$-forms) by resolutions of singularities
and cohomological methods (for $p=2$, but the same method surely gives
the more general results).
By a standard technique this theorem implies global
results for a Stein space $X$. 
In case $X_{sing}$ is a single point  more precise result
are  obtained in \cite{PS} and \cite{FOV1}. In particular, if $\phi$ has bidegree
$(0,q)$, $q<\dim X$, then the image of $L^2(X_{reg})$ under $\dbar$ 
has finite codimension in $L^2(X_{reg})$.
See also \cite{OV}, and the references given there,
for related results.
In \cite{FG}, Forn\ae ss and Gavosto show that, for complex curves,
a H\"{o}lder continuous solution exists if the right hand side  is bounded. 
Special  hypersurfaces and certain homogeneous varieties have  been considered, e.g., in \cite{RuppDiss}
and \cite{RZII}.

\smallskip

We can use our integral formulas to  
solve the  $\dbar$-equation  with compact support. 
As usual this leads to a Hartogs result in $X$,  and a
vanishing result in the complement of a Stein compact,
for forms with not too high degree. The  vanishing result is  well-known
but we can provide a description of the obstruction in the
``limit'' case.
For a given analytic space $X$, let
$\nu=\nu(X)$ be the minimal depth of the local rings $\Ok_{x}$
(the homological codimension).
Since $X$ has pure dimension, $\nu\ge 1$,  and
$X$ is Cohen-Macaulay if and only if $\nu=n$.

\begin{thm}\label{portensats}
Assume  that $X$ is a  connected Stein space of pure dimension $n$
with globally irreducible components $X^\ell$ 
and let $K$ be  a compact subset such that $X_{reg}^\ell\setminus K$ is connected
for each $\ell$.

\smallskip
\noindent (i)\quad If $\nu\ge 2$, then for each holomorphic
function $\phi\in\Ok(X\setminus K)$ there is
$\Phi\in\Ok(X)$ such that $\Phi=\phi$ in $X\setminus K$.

\smallskip
\noindent(ii)\quad Assume that  $\nu=1$ and let  $\chi$ be  a cutoff function
that is identically $1$ in a \nbh of $K$ and with support in a relatively
compact Stein space  $X'\subset\subset X$. 
There is an almost semi-meromorphic $\dbar$-closed $(n,n-1)$-current $\omega_{n-1}$  on $X'$ that
is smooth on $X'_{reg}$ such that the function
$\phi\in\Ok(X\setminus K)$ has a holomorphic extension $\Phi$
across $K$ if and only if
\begin{equation}\label{moment}
\int_X\dbar\chi\w \omega_{n-1} \phi h=0, \quad h\in\Ok(X).
\end{equation}
\end{thm}

Part (i) is proved in \cite[Ch. 1 Corollary~4.4]{BS}.
If  $X$ is normal and $X\setminus K$ is connected, then the conditions
of Theorem~\ref{portensats} (i) are fulfilled. If $X$ is not normal it is necessary to assume that
$X^{\ell}_{reg}\setminus K$ is connected; see Example~\ref{hartogsex} in Section~\ref{compsupp} below.
See \cite{OV2} for a further discussion.
For related results proved by other methods see, e.g., \cite{MePo}, \cite{RCrelle}, and \cite{RuppHartog}.

The current  $\omega_{n-1}$ is the top degree component of a
{\it structure  form}  $\omega$ associated to $X$, see Section~\ref{res}.
Since  $\omega_{n-1}$ is {\it almost semi-meromorphic}, see Section~\ref{res}
and  \cite{AS2},
the integrals (the action of $\omega_{n-1}$ on  test forms)
exist as principal values at $X_{sing}$.
If the holomorphic extension 
$\Phi$ exists, then, since $\dbar\omega_{n-1}=0$, we have that 
$$
\int_X\dbar\chi\w \omega_{n-1} \phi h=
\int_X\dbar\chi\w \omega_{n-1} \Phi h=-\int_X \chi \dbar(\omega_{n-1} \Phi h)=0,
$$
and hence  condition \eqref{moment} is necessary;   see, e.g.,  \cite{AS2} for a
discussion on  currents on a singular space.

\smallskip
There is  a similar result for $\dbar$-closed forms (currents) in $\A$:  

\begin{thm}\label{nyporten}  
Let  $X$ be  a Stein space of pure dimension $n$ and let $K\subset X$ be a Stein compact. 
Assume  that   $\phi\in\A_q(X\setminus K)$ and 
$\dbar\phi=0$, and let $X'\subset\subset X$ be a Stein \nbh of $K$.

\smallskip

\noindent (i) If $q\le \nu-2$, then  there is $\Phi\in\A_q(X)$ such that
$\dbar\Phi=0$ and $\Phi=\phi$ outside $X'$.

\smallskip

\noindent (ii) If $q=\nu-1$, then there is such a $\Phi$ if and only if
\begin{equation}\label{moment2}
\int_X\dbar\chi\w \omega_{n-\nu} \w \phi h=0, \quad h\in\Ok(X).
\end{equation}
\end{thm}


As usual this leads to a vanishing theorem for $\dbar$ in $X\setminus K$.

\begin{cor}\label{nyportencor} Assume that $\phi\in\A_q(X\setminus K)$ and $\dbar\phi=0$.

\smallskip
\noindent
(i) If $1\le q\le\nu-2$,  then
there is $\psi\in\A_{q-1}(X\setminus K)$ such that $\dbar\psi=\phi$.

\smallskip
\noindent
(ii) If $1\le q=\nu-1$, then there is  $\psi\in\A_{q-1}(X\setminus K)$ such that 
$\dbar\psi=\phi$ if and only if
\eqref{moment2} holds.
\end{cor}

In view of the exactness of \eqref{resolution},  part (i) is equivalent to that
$H^q(X\setminus K,\Ok)=0$ for $q\le\nu-2$; this vanishing is well-known,
see, e.g., \cite[Section~2]{OV2}. The novelty here is the proof with integral
formulas. Part (ii) provides a representation of the
cohomology for  $q=\nu-1$.

\begin{remark}
It follows from the proofs, and the semicontinuity of $x\mapsto\depth\Ok_x$
that these theorems
hold with  $\nu=\nu(K):=\min_{x\in K} \depth \Ok_{x}$.
In Theorem~\ref{portensats} however, one must take the minimum over
a Stein \nbh of $K$, cf., \cite[footnote on p.~2]{OV2}.
\end{remark}

\smallskip

In the same way we can obtain the existence of $\dbar$-closed
extensions across $X\setminus A$ for any analytic, not necessarily pure dimensional,
subset $A\subset X$, see Proposition~\ref{badanka3} below.
For instance $A$ may be $X_{sing}$. 
This leads to vanishing results in $X\setminus A$.

\begin{thm}\label{main3}
Assume that $X$ is a Stein space of pure dimension $n$,
and let $A$ be an analytic subset of dimension $d\ge 1$.
Assume that $\phi\in\A_q(X\setminus A)$ and $\dbar\phi=0$.

\smallskip
\noindent
(i) If $1\le q \le \nu-2- d$, then there is a $\psi\in\A_{q-1}(X\setminus A)$
such that $\dbar\psi=\phi$.

\smallskip
\noindent
(ii) If $1\le q=\nu-1-d$, then the same conclusion holds if and only if 
\begin{equation}\label{moment3}
\int_X \dbar\chi \w \omega_{n-\nu}\w \phi\w h=0
\end{equation}
for all smooth $\dbar$-closed $(0,d)$-forms $h$ such that the
$\supp h\cap \supp\dbar\chi$ is compact.

\smallskip
If $q=0\le \nu-2- d$ or $q=0=\nu-1$ and \eqref{moment3} holds, then
the conclusion is that $\phi$ is holomorphic and has a holomorphic extension across
$A$.
\end{thm}

Even in this case it is enough to take
$\nu=\nu(A)$. 
Because of the exactness of \eqref{resolution},
part (i) is equivalent to the vanishing of $H^q(X\setminus A,\Ok)$ for $1\le q\le \nu-2-d$,
also this vanishing result is well-known, see \cite{Scheja2}, \cite{Traut}, and \cite{ST}.

\smallskip

In \cite{AS2} we introduced the sheaves $\W_{p,q}$ of pseudomeromorphic 
$(p,q)$-currents on $X$ with the so-called {\it standard extension property} SEP.
It is proved  that the operators $\K$ and $\Pr$ in
Theorem~\ref{main}  extend to operators
$$
\W_{0,q+1}(X)\to\W_{0,q}(X'), \quad \W_{0,0}(X)\to\Ok(\Omega').
$$
Moreover,  the Koppelman formulas hold if, in addition, $\phi$ is 
in the domain $\Dom\dbar_X$ of the operator $\dbar_X$ introduced
in \cite{AS2}.  
The latter condition means that 
 $\dbar\phi$ is in $\W_{0,*}$ and that $\phi$ satisfies a certain ``boundary condition'' at $X_{sing}$.
If $\phi\in\W_{0,0}$,  then  
$\phi$ is in $\Dom\dbar_X$ and $\dbar_X\phi=0$ if and only if
$\phi$ is (strongly holomorphic), whereas 
$\dbar\phi=0$ means that $\phi$ is weakly holomorphic in the
sense of Barlet-Henkin-Passare, cf., \cite{HP}.

We will  mainly be interested here in the case
when $X$ is Cohen-Macaulay. Then  we can always
choose (at least semi-globally) a structure form $\omega$ that only
has one component $\omega_0$ that is a $\dbar$-closed
$(n,0)$-form (current). The
condition $\phi\in\Dom\dbar_X$ then precisely means that
there is a current $\psi$ in $\W_{0,q+1}$ such that 
$$
\dbar(\phi\w\omega)=\psi\w\omega.
$$
For other equivalent conditions, see Section~\ref{res} and \cite{AS2}.

\smallskip

Thus
$\dbar\K\phi=\phi$ in $X'$ if $\phi\in\W_{0,q}(X)\cap\Dom\dbar_X$ 
and $\dbar_X\phi=0$.
Unfortunately we do not know whether $\K\phi$ is again in
$\Dom\dbar_X$; if it were, then  $\W_{0,k}\cap\Dom\dbar_X$
would provide a (fine) resolution of $\Ok$. It is however true, \cite{AS2},
that if $\phi\in\W_{0,0}$ and $\dbar_X\phi=0$, then $\phi\in\Ok$.
Moreover,  the difference of two of our solutions is
anyway $\dbar$-exact on $X_{reg}$ if $q>1$ and strongly holomorphic if $q=1$.
By  an elaboration of these facts we can prove:

\begin{thm}\label{mainglobal}
Assume that $X$ is an analytic space of pure dimension  $n$ and that
$X$ is Cohen-Macaulay. 
Any  $\dbar$-closed $\phi\in\W_{0,q}(X)\cap\Dom\dbar_X$,  $q\ge 1$,  
that is  smooth on $X_{reg}$ 
defines a canonical class in $H^q(X,\Ok^X)$;  if this class
vanishes then there is a global smooth form $\psi$ on $X_{reg}$
such that $\dbar\psi=\phi$. 
In particular, there is  such a solution if
$X$ is a  Stein space.
\end{thm}

\begin{remark} If $\phi$ is not smooth, the conclusion is that there is a 
form $\psi\in\W_{q-1}(X)$ such that $\dbar\psi=\phi$ on $X_{reg}$.

A similar statement holds even if $X$ is not Cohen-Macaulay. However,
the proof then requires a hypothesis on $\phi$ that is
(marginally) stronger than the  $\Dom\dbar_X$-condition,
see Section~\ref{mg}.
\end{remark}

The starting point is a 
certain residue current $R$, introduced in  \cite{AW1}, 
that is   associated to a subvariety $X\subset\Omega$,
and the  integral representation formulas from \cite{A7}.
We discuss the current $R$, and its   associated 
structure form $\omega$ on $X$,   in Section~\ref{res}, 
and in  Section~\ref{koppsec} we recall from \cite{AS2} 
the construction of  the Koppelman formulas.

In Section~\ref{exsec} we describe some concrete realizations
of the ``moment''  condition \eqref{moment} 
in Theorem~\ref{portensats}.  
The remaining sections are devoted to the proofs.

\smallskip
{\bf Acknowledgement:} We are indebted to Jean Ruppenthal and Nils
\O vrelid for important remarks on an earlier version of this paper. We are also grateful to
the anonymous referee for  careful reading and valuable comments.

\section{A residue current associated to $X$}\label{res}

Let $X$ be a subvariety  of pure dimension $n$
of a pseudoconvex set $\Omega\subset\C^N$. 
The Lelong current $[X]$ is a  classical 
analytic object that represents $X$.
It is a $d$-closed $(p,p)$-current, $p=N-n$,  such that 
$$
[X].\xi =\int_X\xi
$$
for test forms $\xi$. If $\codim X=1$,  $X=\{f=0\}$ and  $df\neq 0$ on $X_{reg}$, then
the Poincare-Lelong formula states that 
\begin{equation}\label{pl}
\dbar\frac{1}{f}\w \frac{df}{2\pi i}=[X].
\end{equation}
To construct integral formulas  we will use an analogue of the current $\dbar(1/f)$, 
introduced in \cite{AW1}, for a
general variety $X$. It turns out that this current,
contrary to $[X]$,   also reflects certain  subtleties of the variety
at $X_{sing}$ that are encoded by the  algebraic description of $X$.
Let  $\J$ be the ideal sheaf over $\Omega$ generated by the variety $X$.
In a slightly smaller set, still denoted $\Omega$, one can find a free resolution 
\begin{equation}\label{acomplex}
0\to \Ok(E_M)\stackrel{f_M}{\longrightarrow}\ldots
\stackrel{f_3}{\longrightarrow} \Ok(E_2)\stackrel{f_2}{\longrightarrow}
\Ok(E_1)\stackrel{f_1}{\longrightarrow}\Ok(E_0)
\end{equation}
of the sheaf $\Ok/\J$.
Here $E_k$ are trivial vector bundles over $\Omega$
and $E_0=\C$ is a trivial line bundle.
This resolution induces a complex  of trivial vector bundles
\begin{equation}\label{bcomplex}
0\to E_M\stackrel{f_M}{\longrightarrow}\ldots
\stackrel{f_3}{\longrightarrow} E_2\stackrel{f_2}{\longrightarrow}
E_1\stackrel{f_1}{\longrightarrow}E_0\to 0
\end{equation}
that is pointwise exact outside $X$.

Let $\nu=\nu(X)$ be the minimal depth of the rings $\Ok^\Omega_x/\J_x=\Ok^X_x$.
Then there is a resolution \eqref{acomplex}  with $M=N-\nu$.
Since $\nu\ge 1$ we may thus  assume that $M\le N-1$.
If (and only if) $X$ is Cohen-Macaulay, i.e., all the rings $\Ok^X_x$
are Cohen-Macaulay, there is a resolution \eqref{acomplex} with $M=N-n$.

Given Hermitian metrics on  $E_k$, in \cite{AW1}  was defined  a 
current $U=U_1+\cdots+U_M$, where $U_k$ is a $(0,k-1)$-current 
that is smooth outside $X$ and takes values in
$E_k$,  and a residue current with support on $X$,
\begin{equation}\label{rd}
R=R_p+R_{p+1}+\cdots+R_M,
\end{equation}
where $R_k$ is a $(0,k)$-current with values in $E_k$,
satisfying
\begin{equation}\label{plutt}
\nabla_f U=1-R,
\end{equation}
and   $\nabla_f=f-\dbar= \sum f_j-\dbar$.

Let $F=f_1$.
The form-valued functions  $\lambda\mapsto |F|^{2\lambda}u=:U^\lambda$ 
(here $u$ is the restriction of $U$ to $\Omega\setminus  X$) and 
$1-|F|^{2\lambda}+\dbar|F|^{2\lambda}\w u=:R^\lambda$,
a~priori defined for $\Re\lambda>>0$, admit  analytic continuations as current-valued
functions  to $\Re\lambda>-\epsilon$ and
\begin{equation}\label{snabel}
U=U^\lambda |_{\lambda=0},  \quad  R=R^\lambda|_{\lambda=0}.
\end{equation}
Notice also that $\nabla_f U^\lambda=1-R^\lambda$.

It is proved in \cite{AS2} that $R$ has the {\it standard extension property}, SEP, with respect to $X$.
This means that 
 if $h$ is a holomorphic function that does not vanish identically
on any component of $X$ (the most interesting case is when  $\{h=0\}$ contains
$X_{sing}$), 
$\chi$ is a smooth approximand of the characteristic
function for $[1,\infty)$, and  $\chi_\delta=\chi(|h|/\delta)$, then
\begin{equation}\label{sepp}
\lim_{\delta\to 0}\chi_\delta R=R.
\end{equation}
The SEP  can also be expressed as saying that
$R$ is equal to the value at $\lambda=0$,  $|h|^{2\lambda}R|_{\lambda=0}$,
of (the analytic continuation of) $\lambda \mapsto |h|^{2\lambda}R|_{\lambda=0}$, see, e.g., \cite{AW2}.

It holds that $\nabla_f\circ\nabla_f=0$, and in view of \eqref{plutt}, thus 
$\nabla_f R=0$, so in particular, $\dbar R_M=0$.

\smallskip
We say that a current $\mu$ on $X$ has the SEP on $X$ if (with $\chi_\delta$ as above)
$\chi_\delta \mu\to \mu$ when $\delta\to 0$, for each holomorphic
$h$ that does not vanish identically on any irreducible component of $X$.
We recall from \cite{AS2} that a current $\mu$ on $X$ is {\it almost semi-meromorphic}
if it is the direct image of a semi-meromorphic current under a modification
$\tilde X\to X$, see, \cite{AS2}. Such a current $\mu$ 
is pseudomeromorphic  and has the SEP on $X$, 
so in particular it is in  $\W$.

It is proved in \cite{AS2} that there is a (unique) almost semi-meromorphic current
\begin{equation*}
\omega=\omega_0 + \omega_1 + \cdots + \omega_{n+M-N}
\end{equation*}
on $X$, where $\omega_{r}$ has bidegree $(n,r)$ and takes values in $E^r:=E_{N-n+r}|_X$, such that 
\begin{equation}\label{DLrep}
i_* \omega = R\wedge dz_1\wedge \cdots \wedge dz_N.
\end{equation}
The current $\omega$ is smooth and nonvanishing (\cite[Lemma~18]{AS2})
on $X_{reg}$ and 
\begin{equation}\label{omegauppsk}
|\omega| = \Ok (\delta^{-M})
\end{equation}
for some $M\ge 0$, where $\delta$ is the distance to $X_{sing}$.
We say that $\omega$ is a {\em structure form} for $X$, cf., Remark~\ref{bulla} below.
The equality \eqref{DLrep} means that
$$
\int_\Omega R\w  dz_1\wedge \cdots \wedge dz_N\w \xi = \int_X \omega\w \xi
$$
for each test form $\xi$ in $\Omega$. Here both integrals mean currents acting
on the test form;  the right  hand side can also be interpreted as the principal value
$$
\lim_{\delta\to 0}\int_X \chi_\delta \omega\w\xi.
$$
In particular it follows that for a smooth form  $\Phi$,
$R\w\Phi$ only depends on the pull-back of $\Phi$ to  $X_{reg}$.

\begin{remark}\label{bulla}
Let 
\begin{equation*}
E^r:=E_{p+r} |_{X}, \quad f^r:=f_{p+r}|_X
\end{equation*} 
so that $f^r$ becomes a holomorphic section of $\mbox{Hom}\, (E^{r}, E^{r-1})$.
Then $\nabla_f=f^\bullet-\dbar$ has a meaning on $X$. If $\phi$ is a meromorphic function,
or even $\phi\in\W_{0,0}$ on $X$, then $\phi\w\omega$ is a well-defined
current in $\W$ and $\phi$ is strongly holomorphic if and only if
\begin{equation}\label{pruta}
\nabla_f(\phi\w\omega)=0.
\end{equation}
If $X$ is Cohen-Macalay and $\omega=\omega_0$, then 
\eqref{pruta}  precisely means that $\dbar(\phi\w\omega)=0$
(which by definition means that $\phi$ is in $\Dom\dbar_X$ and $\dbar\phi=0$).
In this case $\dbar \omega_0=0$, i.e., 
$\omega_0$ is a weakly holomorphic (in the Barlet-Henkin-Passare sense) $(n,0)$-form;  
thus $\omega_0$ is precisely so singular it possibly can  be 
and still be $\dbar$-closed.
\end{remark}

From the proof of Proposition~16 in \cite{AS2} it follows that we can write
$R = \gamma \lrcorner [X]$, 
where $\gamma=\gamma_0+\cdots \gamma_{n-1}$ is smooth in $\Omega\setminus X_{sing}$,
almost semi-meromorphic in $\Omega$, and $\gamma_r$ takes values in 
$E_{p+r}\otimes T^*_{0,r}(\Omega)\otimes \Lambda^p T_{1,0}(\Omega)$. 
In view of \eqref{DLrep} it follows that
\begin{equation}\label{gam}
\int_X \omega\wedge \xi = \int R\w d\zeta\w\xi=
\pm\int_X (\gamma \lrcorner d\zeta) \wedge \xi, \quad
\xi\in \mathcal{D}_{0,*}(X),
\end{equation}
so  in particular,  $\omega=\pm\gamma\lrcorner d\zeta$.

\section{Construction of  Koppelman formulas on $X$}\label{koppsec}

Some of the material in this section overlap with \cite{AS2} but it is  included here for the
reader's convenience and to make the proof of Theorem~\ref{mainglobal} more
accessible. 
We first  recall the construction of integral formulas in \cite{A1} 
on an open set $\Omega$ in $\C^N$.
Let $(\eta_1,\ldots,\eta_N)$ be a holomorphic tuple in $\Omega_\zeta\times\Omega_z$
that span the ideal associated to the diagonal 
$\Delta\subset\Omega_\zeta\times\Omega_z$.
For instance, one can take $\eta=\zeta-z$. 
Following the last section in \cite{A1}  we 
consider forms in $\Omega_\zeta\times \Omega_z$ with values
in the exterior algebra $\Lambda_\eta$ 
spanned by $T^*_{0,1}(\Omega\times \Omega)$ and
the $(1,0)$-forms $d\eta_1,\ldots,d\eta_N$. 
On such forms interior multiplication $\delta_\eta$ with
$$
\eta=2\pi i \sum_1^N\eta_j\frac{\partial}{\partial \eta_j}
$$
has a meaning. We then introduce  $\nabla_{\eta}=\delta_{\eta}-\dbar$,
where $\dbar$ acts\footnote{For the time being, also $d\eta_j$ is supposed to include differentials
with respect to both $\zeta$  and  $z$; however, at the end only 
the $d_{\zeta}\eta_j$ come into play in this paper.} on both $\zeta$ and $z$. 
Let $g=g_{0,0}+\cdots +g_{N,N}$   be a smooth form (in $\Lambda_\eta$)
defined for $z$ in  $\Omega'\subset\subset \Omega$  and $\zeta \in \Omega$, such that
$g_{0,0}=1$ on the diagonal $\Delta$ in $\Omega'\times \Omega$ and $\nabla_\eta g=0$.
Here and in the sequel lower index $(p,q)$ denotes bidegree. Since $g$ takes values in
$\Lambda_\eta$ thus $g_{k,k}$ is the term that has degree $k$ in $d\eta$.
Such a form $g$ will be called a {\it weight} with respect to $\Omega'$.
Notice that if $g$ and $g'$ are weights, then $g\w g'$ is again  a  weight.

\begin{ex}\label{alba}
If $\Omega$ is pseudoconvex and $K$ is a holomorphically convex compact
subset, then  one can find a weight with respect to some \nbh $\Omega'$ of $K$,
depending holomorphically on $z$,  
that has compact support (with respect to $\zeta$) 
in $\Omega$, see, e.g., \cite[Example~2]{A7}.
Here is an explicit choice  when $K$ is the closed ball $\overline{\B}$
and $\eta=\zeta-z$:
If 
$\sigma=\bar\zeta\cdot d\eta/2\pi i(|\zeta|^2-\bar\zeta\cdot z)$, 
then   $\delta_{\eta}\sigma=1$ for $\zeta\neq z$ and 
$$
\sigma\w(\dbar \sigma)^{k-1}=\frac{1}{(2\pi i)^k}
\frac{\bar\zeta\cdot d\eta\w(d\bar\zeta\cdot d\eta)^{k-1}}
{(|\zeta|^2-\bar\zeta\cdot  z)^k}.
$$
If $\chi$ is  a cutoff function that is $1$ in a slightly
larger ball, then we can take  
$$
g=\chi-\dbar\chi\w\frac{\sigma}{\nabla_{\eta} \sigma}=
\chi-\dbar\chi\w [\sigma+\sigma\w\dbar \sigma+ \sigma\w(\dbar \sigma)^2+
\cdots +\sigma\w(\dbar \sigma)^{N-1}].
$$
Observe that  $1/\nabla_\eta \sigma=1/(1-\dbar \sigma)=1+\dbar \sigma +(\dbar \sigma)^2+\cdots$.
One can find a $g$ of the same form in the general case. 
\end{ex}

Let $s$ be a smooth $(1,0)$-form in $\Lambda_\eta$ such that $|s|\le C|\eta|$ and
$|\delta_\eta s|\ge C|\eta|^2$; such an $s$ is called  {\it admissible}. Then 
$B=s/\nabla_\eta s$ is a locally integrable form and 
\begin{equation}\label{bm}
\nabla_\eta B=1-[\Delta],
\end{equation}
where $[\Delta]$ is the $(N,N)$-current of integration over the diagonal
in $\Omega\times \Omega$. 
More concretely,  
$$
B_{k,k-1}=\frac{1}{(2\pi i)^k}\frac{s\w(\dbar s)^{k-1}}{(\delta_\eta s)^k}.
$$
If $\eta =\zeta-z$, 
$s=\partial|\eta|^2$ will do, and we then  refer to the  resulting form $B$
as the Bochner-Martinelli form. In this case
$$
B_{k,k-1}=\frac{1}{(2\pi i)^k}
\frac{\partial|\zeta-z|^2\w(\dbar\partial|\zeta-z|^2)^{k-1}}{|\zeta-z|^{2k}}.
$$

\smallskip
Assume now that $\Omega$ is pseudoconvex.
Let us fix global frames for the bundles $E_k$ in \eqref{bcomplex} over $\Omega$.
Then $E_k\simeq\C^{\rank E_k}$,  and 
the morphisms $f_k$  are just matrices of holomorphic functions.
One can find (see \cite{A7} for explicit choices)
$(k-\ell,0)$-form-valued Hefer morphisms, i.e.,  matrices,  
$H^\ell_k\colon E_k\to E_\ell$,    depending holomorphically on $z$ and $\zeta$, such that 
$H^\ell_k=0$ for $ k<\ell$, $ H^\ell_\ell=I_{E_\ell}$, and in general,
\begin{equation}\label{Hdef}
\delta_{\eta} H^\ell_{k}=
H^\ell_{k-1} f_k -f_{\ell+1}(z) H^{\ell+1}_{k};
\end{equation}
here $f$ stands for $f(\zeta)$. Let
$$
HU = \sum_k H^{1}_k U_k,
\quad  HR=\sum_k H^0_k R_k.
$$
Thus $HU$ takes a section $\Phi$ of $E_0$, i.e., a function,
depending on $\zeta$ into a (current-valued) 
section $HU\Phi$ of $E_{1}$ depending on both $\zeta$ and $z$, and similarly,
$HR$ takes a section of $E_0$ into a section of $E_0$.
We can have 
$$
g^\lambda= f(z)HU^\lambda + HR^\lambda
$$
as smooth as we want by just taking $\Re\lambda$ large enough.
If  $\Re\lambda>>0$, then, cf.,  \cite[p.\ 235]{A7},
$g^\lambda$ is a weight,  and in view of \eqref{bm} thus
$$
\nabla_\eta(g^\lambda\w g\w B)=g^\lambda\w g-[\Delta]
$$
from which we get  
\begin{equation*}
\dbar(g^\lambda\w g\w B)_{N,N-1}=[\Delta]-(g^\lambda\w g)_{N,N}.
\end{equation*}
As in
\cite{A7} we  get the Koppelman formula 
\begin{equation}\label{balja}
\Phi(z)=\int_\zeta (g^\lambda\w g\w B)_{N,N-1}\w\dbar\Phi +
\dbar_z\int_\zeta (g^\lambda\w g\w B)_{N,N-1}\w\Phi +\int_\zeta 
 (g^\lambda\w g)_{N,N}\w\Phi
\end{equation}
for $z\in \Omega'$, and since $g^\lambda=HR^\lambda$ when $z\in X_{reg}$ we
get
\begin{multline}\label{balja2}
\Phi(z)=\int_\zeta (HR^\lambda \w g\w B)_{N,N-1}\w\dbar\Phi +\\
\dbar_z\int_\zeta  (HR^\lambda\w g\w B)_{N,N-1}\w\Phi +\int_\zeta 
  (HR^\lambda\w g)_{N,N}\w\Phi,\quad z\in X_{reg}'.
\end{multline}
It is proved in \cite{AS2}, see also \cite{AS1} for a slightly different
argument, that we can put $\lambda=0$ in \eqref{balja2} and thus
$$
\Phi(z)=\K\dbar\Phi +\dbar\K\Phi+\Pr\Phi, \quad z\in X'_{reg},
$$
where 
\begin{equation}\label{kurt5}
\K\Phi(z)=\int_\zeta(HR\w g\w B)_{N,N-1}\w\Phi, \quad z\in X_{reg}',
\end{equation}
and
\begin{equation}\label{kurt8}
\Pr\Phi(z)=\int_\zeta (HR\w g)_{N,N}\w \Phi, \quad z\in \Omega'.
\end{equation}
If $\Phi$ is vanishing in a \nbh of some given point
$x$ on $X_{reg}$, then $B\w\Phi$ is smooth in $\zeta$ for $z$ close to $x$,  and 
the integral in \eqref{kurt5} is to be  interpreted as the current
$R$ acting on a smooth form. It is clear that this  integral depends
smoothly on $z\in X'_{reg}$. 
Notice that 
\begin{multline*}
(HR\w g\w B)_{N,N-1}=\\
H_p^0R_p\w(g\w B)_{N-p,N-p-1}+H_{p+1}^0R_{p+1}\w (g\w B)_{N-p-1,N-p-2}+\cdots,
\end{multline*}
cf., \eqref{rd}, 
and that
\begin{equation}\label{modehatt}
(g\w B)_{N-k,N-k-1}=\Ok(1/|\eta|^{2N-2k-1})
\end{equation}
so it is integrable on $X_{reg}$ for $k\ge N-n$.
If $\Phi$ has support close to $x$, therefore \eqref{kurt5} has a meaning as an approximative
convolution and is again smooth in $z\in X_{reg}$ according to  Lemma~\ref{brum}  below.

From Section~\ref{res} is is clear that these formulas only depend on the
pullback $\phi$ of $\Phi$ to $X_{reg}$, and in view of \eqref{gam} we  have

\begin{prop}\label{koppelmanthm} Let  $g$ be  any smooth weight in $\Omega$ with 
respect to $\Omega'$ and with
compact support in $\Omega$. For any smooth $(0,q)$-form $\phi$ on $X$,
$\K\phi$ is a smooth $(0,q-1)$-form in $X_{reg}'$, $\Pr\phi$ is a smooth
$(0,q)$-form in $\Omega'$, and we have the  Koppelman formula
\begin{equation}\label{koppelman3} 
\phi(z)=\dbar\K\phi(z)+\K(\dbar\phi)(z)+\Pr\phi(z), \quad z\in X'_{reg}.
\end{equation}
where 
\begin{equation}\label{kpv2}
\K\phi(z)=\int_\zeta k(\zeta,z)\w \phi(\zeta),
\quad
\Pr\phi(z)=\int_\zeta p(\zeta,z)\w \phi(\zeta),
\end{equation}
and
\begin{equation}\label{kpv3}
k(\zeta,z):=\pm\gamma\lrcorner(H\w g\w B)_{N,N-1}, \quad p(\zeta,z):=\pm\gamma\lrcorner(H\w g)_{N,N}.
\end{equation}
\end{prop}

Since $B$ has bidegree $(*, *-1)$,
 $\K\phi$ is a $(0,q-1)$-form and $\Pr \phi$ is $(0,q)$-form.
It follows from \eqref{sepp} that the integrals in \eqref{kpv2} exist
as principal values at $X_{sing}$,  i.e.,  $\K\phi=\lim\K(\chi_\delta\phi)$
and $\Pr\phi=\Pr(\chi_\delta\phi)$ if $\chi_\delta$ is as in \eqref{sepp}.

From \eqref{omegauppsk} and \eqref{gam} we find that
\begin{equation}\label{gam69}
k(\zeta,z)=\omega(\zeta)\w\alpha(\zeta,z)/|\eta|^{2n},
\end{equation}
where $\alpha$ is a smooth form that is  $\Ok(|\eta|)$.

\begin{remark}\label{zoran}
Assume that $\phi$ is (smooth on $X_{reg}$ and) in $\W_{0,q}(X)$. Then, see \cite{AS2},
$\K\phi$ and $\Pr\phi$ still define elements in  $\W(X')$ that are smooth in $X_{reg}'$.
Assume  that $\phi$ in addition is in $\Dom\dbar_X$. This means (implies) that
$
\dbar\chi_\delta \w\phi\w\omega\to 0.
$
Applying \eqref{kpv2} to $\chi_\delta\phi$ for $z\in X_{reg}'$ and letting $\delta\to 0$,
we conclude that  \eqref{kpv2}  holds for $\phi$ as well.  In particular,
$\dbar\K\phi=\phi$ if $\dbar\phi=0$.
\end{remark}

\begin{remark}\label{adef}
In \cite{AS2} we defined  $\A$ as the smallest sheaf that is closed under multiplication with smooth
forms and the action of any operator $\K$ as above with a weight $g$ that is holomorphic
in $z$. We can just as well admit any smooth weight $g$ in the definition. The basic
Theorem~2 in \cite{AS2} holds also for this  
possibly slightly larger sheaf, that we still denote by $\A$. Basically
the same proof works; the only difference is that in
\cite[(7.2)]{AS2} we get an additional smooth term $\Pr\phi_{\ell-1}$, which however
does not affect the conclusion.
With this wider definition of $\A$ we have that 
$\K$ and $\Pr$ in \eqref{kpv2} extend to operators
$\A(X)\to\A(X')$ and $\A(X)\to\E_{0,*}(X')$, respectively.
\end{remark}

\begin{lma}\label{brum}
Suppose that $V\subset \Omega$ is smooth with codimension $p$ and 
$\xi$ has compact support and $\nu\le N-p$. If $\xi$ is in 
$C^k(V)$, then 
$$
h(z)=\int_{\zeta\in V} \frac{(\bar\zeta_i-\bar z_i)\xi(\zeta)}
{|\zeta-z|^{2\nu}}
$$
is in $C^k(V)$ as well for $i=1,\ldots,N$.
\end{lma}



\section{Proofs of Theorems~\ref{main},  \ref{maincor}, and \ref{main2}}

\begin{proof}[Proof of Theorem~\ref{main}]
If we  choose $g$ as the weight from Example~\ref{alba} then 
$\Pr \phi$ will vanish for degree reasons unless $\phi$ has bidegree
$(0,0)$, i.e., is a function, and in that case clearly
$\Pr \phi$ will be holomorphic for all $z$ in $\Omega'$.
Now Theorem~\ref{main} follows from 
\eqref{koppelman3} except for the asymptotic estimate
\eqref{asymp}.

After a slight regularization we may assume that
$\delta(z)$ is smooth on $X_{reg}$ or alternatively we can replace 
$\delta$ by  $|h|$ where 
$h$ is a tuple of functions in $\Omega$ such that $X_{sing}=\{h=0\}$,
by virtue of Lojasiewicz' inequality,
\cite{Loja} and \cite{Ma}.
In fact, there is a number $r\ge 1$ such that
\begin{equation}\label{loja}
(1/C) \delta^r(\zeta)\le |h(\zeta)|\le C\delta(\zeta).
\end{equation}

We have to estimate, cf., \eqref{gam69},  
\begin{equation}\label{puck}
\int_\zeta \omega(\zeta)\w\frac{\alpha(\zeta,z)}{|\eta|^{2n}}
\end{equation}
when $z\to X_{sing}$.
To this end we take a smooth approximand $\chi$ of $\chi_{[1/4,\infty)}(t)$
and write \eqref{puck}  as
\begin{multline*}
\int_\zeta \chi(\delta(\zeta)/\delta(z))
\omega(\zeta)\w \frac{\alpha(\zeta,z)}{|\eta|^{2n}}+ 
\int_\zeta\big(1-\chi(\delta(\zeta)/\delta(z))\big)
\omega(\zeta)\w \frac{\alpha(\zeta,z)}{|\eta|^{2n}}.
\end{multline*}
In the first integral,  $\delta(\zeta)\ge C \delta(z)$ and since
the integrand is
integrable we can use \eqref{omegauppsk} and get the estimate
$\lesssim \delta(z)^{-M}$ for some $M$.
In the second integral we use instead that
$\omega$ has some fixed finite order as a current so that its action can be estimates by
a finite number of derivatives of
$
(1-\chi(\delta(\zeta)/\delta(z)))\alpha(\zeta,z)/|\eta|^{2n},
$
which again is like $\delta(z)^{-M}$ for some $M$,  since
here $\delta(\zeta)\le\delta(z)/2$ and hence 
$
C|\eta|\ge|\delta(z)-\delta(\zeta)|\ge \delta(z)/2.
$
Thus \eqref{asymp} holds.
\end{proof}

\begin{proof}[Proof of Theorem~\ref{maincor}]
Suppose that $\nu$ is the order of the current $R$.
Since $\K\Phi$ basically is the current $R$ acting on
$\Phi$ times a smooth form, it is clear that the Koppelman
formula \eqref{koppelman3},  but with $\Phi$,    remains  true even if $\Phi$ is just of class
$C^{\nu+1}$ in a \nbh of $X$. For  instance, for given $\Phi$ in  $C^{\nu+1}$ 
this follows by approximating in $C^{\nu+1}$-norm by smooth forms.

It is a more delicate matter to check that 
$\K\Phi$ only depends on the pullback of $\Phi$ to $X$. 
The current  $R$ 
is (locally)  the push-forward, under a suitable modification
$\pi\colon Y\to \Omega$,  of a finite sum $\tau=\sum \tau_j$
where each $\tau_j$ is a  simple current of the form 
\begin{equation}\label{pudel}
\tau_j=\dbar\frac{1}{t_{j_1}^{a_{j_1}}}
\w\frac{\alpha_j}{t_{j_2}^{a_{j_2}}\cdots t_{j_r}^{a_{j_r}}},
\end{equation}
with  a  smooth form  $\alpha_j$. Since $R$ has the SEP with respect to $X$,
arguing as in \cite[Section~5]{AW2},
we can assume that the image of each of the  divisors $t_{j_1}=0$ is not 
fully contained in $X_{sing}$. Here is a sketch of a proof:
Write  $\tau=\tau'+\tau''$
where $\tau''$ is the sum of all $\tau_j$ such that the image of  $t_{1_j}=0$
is contained in $X_{sing}$. Let $\chi_\delta=\chi(|h|/\delta)$, where $h$ is a holomorphic
tuple that cuts out $X_{sing}$.  Then $\lim (\pi^*\chi_\delta)\tau''=0$ and $\lim
(\pi^*\chi_\delta) \tau'=\tau'$. Since $R=\pi_*\tau$ and $\lim \chi_\delta R=R$,
it follows that $R=\pi_*\tau'$.

Therefore, if  $i\colon X\to \Omega$ and $i^*\Phi=0$ on $X_{reg}$,
then the pullback of $\pi^*\Phi$ to $t_{1_j}=0$ must vanish.
If $\Phi$ is in $C^{L+1}$, where   $L$ is the maximal sum of the powers
in the denominators in \eqref{pudel}, 
it follows that $\Phi\w R=\pi_*(\pi^*\Phi\w\tau)=0$ and similarly $\dbar\Phi\w R=0$.
\end{proof}


\begin{proof}[Proof of Theorem~\ref{main2}]
We will  use an extra weight factor. 
In a slighly smaller domain $\Omega''\subset\subset\Omega$
we can find a holomorphic tuple $a$ 
such that $\{a=0\}\cap X\cap\Omega''=X_{sing}\cap\Omega''$.
Let  $H^a$ be a holomorphic $(1,0)$-form in $\Omega''\times\Omega''$
such that
$
\delta_\eta H^a=a(\zeta)-a(z).
$
If $\psi$ is a $(0,q)$-form that
vanishes in a \nbh of $X_{sing}$  we can incorporate a suitable power of
the weight  
\begin{equation}\label{viktig}
g_a =\frac{a(z)\cdot \bar a}{|a|^2}+H^a\cdot\dbar\frac{\bar a}
{|a|^2}
\end{equation}
in \eqref{koppelman3}; we  will use  the weight $g_a^{\mu+n}\w g$
instead of just  $g$,  the usual weight with respect to 
$\Omega'\subset\subset\Omega''\subset\subset\Omega$ that has
compact support and is holomorphic in $z$. For degree reasons,
the second term on the right hand side of \eqref{viktig} can occur
to the power at most $n$ when pulled back to $X$, and hence  
the associated kernel
$$
k^\mu(\zeta,z)=\gamma\lrcorner(H\w g^{\mu+n}_a\w g\w B)_{N,N-1}
$$
is like, cf., \eqref{gam}, 
$$
\omega(\zeta)\w \Big(\frac{a(z)\cdot \overline{a(\zeta)}}{|a(\zeta)|^2}\Big)^\mu\w\Ok(1/|\eta|^{2n-1}).
$$
The  operators in  Lemma~\ref{brum}
are  bounded on $L^p_{loc}$, so we have that 
\begin{equation}\label{sallan2}
\psi=\dbar \int_{X_{reg}} k^\mu(\zeta,z)\psi(\zeta) +
\int_{X_{reg}} k^\mu(\zeta,z)\w \dbar\psi(\zeta)
\end{equation}
for $(0,q)$-forms $\psi$,  $q\ge 1$, in $L^p(X_{reg})$
that vanish   in a \nbh of $X_{sing}$. 
If $\phi$  is as in Theorem~\ref{main2}, thus
 \eqref{sallan2} holds for
$\psi=\chi_\epsilon \phi$, where  $\chi_\epsilon= \chi(|a|^2/\epsilon)\phi$ 
and $\chi$ is a smooth approximand of the characteristic function for $[1,\infty)$.

If now  $\mu'\ge M+r+\mu r$, where $M$ is as in \eqref{omegauppsk}
and $r$ as in \eqref{loja},  noting that 
$\dbar\chi_\epsilon\sim 1/|a|$,
it follows that
$$
\int \dbar\chi_\epsilon\w k^\mu\w\phi
$$
tends to zero in $L^p$ when $\epsilon\to 0$ if $\delta^{-\mu'}\phi\in L^p$.
Therefore
$$
u=\int_{X_{reg}} k^\mu(\zeta,z)\w\phi(\zeta)
$$
is a solution such that   $\delta^{-\mu}u\in L^p$.
\end{proof}

\section{Solutions with compact support}\label{compsupp}

The proofs of Theorems \ref{portensats},  \ref{nyporten}, and \ref{main3}
relay on on the 
possibility to solve the $\dbar$-equation with compact support.
To begin with, assume that $X, X',\Omega, \Omega'$  are as in Theorem~\ref{main} and
let $f\in \A_{q+1}(X)$ be $\dbar$-closed and with  support in $X'$.
Choose a resolution \eqref{acomplex} of $\Ok^X=\Ok^\Omega/\J$ in (a slightly smaller set)
$\Omega$ that ends at level $M=N-\nu$ where $\nu$ is the minimal depth of $\Ok^X_{x}$.
Let $\tilde\chi$ be a cutoff function with support in $\Omega'$ that is
identically $1$ in a \nbh of the support of $f$, and
let $g$ be the weight from
Example~\ref{alba}  with this choice of $\tilde\chi$ but with 
$z$ and $\zeta$ interchanged.  
This weight does not have
compact support with respect to $\zeta$, but  since $f$ has compact
support itself we still  have the Koppelman formula \eqref{koppelman3}.
(The one who is worried can include  an extra weight factor with compact support
that is identically $1$ in a \nbh of $\supp\dbar\tilde\chi$;  we are then formally back to the situation
in Proposition~\ref{koppelmanthm}.)
Clearly 
$$
v(z)=\int (HR\w g\w B)_{N,N-1}\w  f
$$
is in $\A_q(X')$ and has support in a \nbh of the support of $f$, 
and it follows from \eqref{koppelman3}  that it is indeed a solution if 
the associated integral $\Pr f$ vanishes. However, since now $\sigma$
is holomorphic in $\zeta$,  for degree reasons we have that
\begin{equation}\label{ulv}
\Pr f(z)=
\pm \dbar\tilde{\chi}(z)\w \int HR_{N-q-1}\w
\sigma \w (\dbar \sigma)^q\w f.
\end{equation}
If  $q\le \nu-2$, this integral vanishes since then $N-q-1\ge N-\nu+1$
so that   $R_{N-q-1}=0$.
If $q=\nu-1$, then $\Pr f(z)$ vanishes if  
\begin{equation}\label{alla}
\int R_{N-q-1}\w d\zeta_1\w\ldots\w d\zeta_N \w f  h=
\pm \int_{X}f\w h \omega_{n-\nu}=0
\end{equation}
for all $h\in\Ok(X')$,  and by approximation
it is enough to assume that \eqref{alla}  holds for $h\in\Ok(X)$.

\begin{remark}
The condition  \eqref{alla} is necessary: Indeed if there is a solution $v\in\A_q(X')$ with compact
support, then since $\dbar \omega_{n-\nu}=0$ in  $X'$ we have that
$$
\int_{X} f\w  h\omega_{n-\nu}=\pm \int_X \dbar v\w h\omega_{n-\nu}=0,
$$
since $\dbar(v\omega_{n-\nu})=\dbar v\w\omega_{n-\nu}$. This in turn holds, since
$
\nabla_f (v\w\omega)=-\dbar v\w\omega,
$
which directly follows from the definition of $v$ being in $\A\subset \Dom\dbar_X$. 
\end{remark}

\begin{proof}[Proof of Theorem~\ref{portensats}]
Since $X$ can be exhausted by holomorphically convex subsets each of which
can be embedded in some affine space, we can assume from the beginning that
$X\subset\Omega\subset\C^N$, where $\Omega$ is holomorphically convex (pseudoconvex).
Let $\Omega'\subset\subset\Omega$ be a holomorphically convex open set in $\Omega$ that
contains  $K$. Let $\chi$ be a cutoff function with support in $\Omega'$ that is
$1$ in a \nbh of $K$ and let $f=\dbar\chi\w\phi$.
Then $(1-\chi)\phi$ is a smooth function in $X$ that
coincides with $\phi$ outside a \nbh of $K$.
As we have seen above, one can find a $u\in\A_0(X)$ with 
support in $X'$ such that $\dbar u=f$ if either
$\nu\ge 2$ or \eqref{alla}, i.e., \eqref{moment},  holds.

Since  $X_{sing}$ is not contained in $K$, our solution
$u$ is, outside of $K$, only smooth  on $X_{reg}$.
Therefore
$\Phi=(1-\chi)\phi+u$ is holomorphic in $X_{reg}$, in a
\nbh of $K$, and outside $\Omega'$. Since
$X_{reg}^\ell\setminus K$ is connected, $\Phi=\phi$ 
there.  
(It follows directly that $\Phi$ is in $\Ok(X)$, since it is in $\A_0(X)$
and $\dbar\Phi=0$.)
\end{proof}

\begin{ex}\label{hartogsex}
Let $X\subset \C^2$ be an irreducible curve with one transverse self intersection at $0\in \C^2$.
Close to $0$, $X$ has two irreducible components, $A_1$, $A_2$, each isomorphic to a disc in $\C$.
Let $K\subset A_1$ be a closed annulus surrounding the intersection point $A_1\cap A_2$. Then 
$X\setminus K$ is connected but $X_{reg} \setminus K$ is not. Denote 
the ``bounded component'' of $A_1\setminus K$ by $U_1$ and put $U_2=X\setminus (K\cup U_1)$.
Let $\tilde{\phi}\in \mathcal{O}(X)$ satisfy $\tilde{\phi}(0)=0$ and define $\phi$ to be $0$
on $U_1$ and equal to $\tilde{\phi}$ on $U_2$. Then $\phi\in \mathcal{O}(X\setminus K)$ and a 
straight
forward verification shows that $\phi$ satisfies the compatibility condition \eqref{moment}.
However, it is clear that $\phi$ cannot be extended to a strongly holomorphic function on $X$.     
\end{ex}

\begin{proof}[Proof of Theorem~\ref{nyporten} and Corollary~\ref{nyportencor}]
Theorem~\ref{nyporten} is proved in pretty much the same way as Theorem~\ref{portensats}. 
Again we can assume that
$X\subset\Omega\subset\C^N$.
Again take $\chi$ that is $1$ in a \nbh of $K$ and with compact support in $X'$.
There is then a solution in $\A_q(X')$ to $\dbar u=\dbar\chi\w \phi$ with support in $X'$
if $q\le \nu-2$ or $q=\nu-1$ and \eqref{alla}, i.e., \eqref{moment2} holds.
Thus $\Phi=(1-\chi) \phi +u$ is in $\A_q(X)$, $\dbar\Phi=0$, and
$\Phi=\phi$ outside $X'$.

\smallskip
Let us now consider the corollary. We may assume that 
$$
K\subset\cdots X_{\ell+1}\subset\subset X_{\ell}\subset\subset \cdots X_0\subset\subset X,
$$
where all $X_\ell$ are  Stein spaces.
It follows from Theorem~\ref{nyporten} that for each $\ell$ there is a
$\dbar$-closed $\Phi_\ell\in\A_q(X)$  that coincides with $\phi$ outside
$X_\ell$, if $q\le \nu-2$ or $q=\nu-1$ and \eqref{moment2} holds.
From the exactness of \eqref{resolution}   we have  $u'_\ell\in\A_{q-1}(X)$ such that
$\dbar u_\ell'=\Phi_\ell$. Since $\dbar (u'_\ell-u'_{\ell+1})=0$ outside
$X_\ell$, there is a $\dbar$-closed $w_\ell\in\A_{q-1}(X)$ such that
$w_\ell=u'_\ell-u'_{\ell+1}$ outside $X_\ell$ (or at least outside $X_{\ell-1}$).
If we let $u_k=u_k'-(w_1+\cdots +w_{k-1})$ then $u=\lim u_k$ exists
and solves $\dbar u=\phi$ in $X\setminus K$.
\end{proof}

One can show directly that the conditions  \eqref{moment} and \eqref{moment2}
are  independent of the choice of metrics on $E_\bullet$:
Let $R'$ and $R$ be the currents correspondning to two different metrics. With the notation in the
proof of Theorem~4.1 in \cite{AW1} we have  
\begin{equation}\label{snara}
\nabla_f M=R-R',
\end{equation}
where
$
M=\dbar|F|^{2\lambda}\w u'\w u|_{\lambda=0}.
$
It follows as in this proof that, outside $X_{sing}$,  $M=\beta R_{N-n}$ where
$\beta$ is smooth. Following the proof of Proposition~16 in \cite{AS2}, we find that
in  fact $M\w dz=i_* m$, where $m=\beta\omega_0$ outside $X_{sing}$. However, $\beta$
is a sum of terms like 
$$(\dbar\sigma'_{n-\nu})\cdots(\dbar\sigma'_{r+1})\sigma'_r
(\dbar\sigma_{r-1})\cdots (\dbar\sigma_{N-n+1}),
$$
it is therefore
almost semimeromorphic on $X$, and thus  $m=\beta\omega_0$.
Moreover, as in the proof of the main lemma \cite[Lemma~27]{AS2} it follows
that $\dbar\chi_\delta\w\beta\w\omega_0\w \phi\to 0$ when
$\delta\to 0$  if $\phi$ is in $\A$. Therefore,
$$
\int_X\dbar m^{N-n}\w \phi\w h =\lim_{\delta\to 0} \int_X  \chi_\delta \dbar m^{N-n}\w \phi\w h 
=
\pm \lim_{\delta\to 0} \int_X  m\w\dbar\chi_\delta\w \phi\ h=0.
$$
From \eqref{snara} we have that $\dbar M_{N-\nu}=R'_{N-\nu}-R_{N-\nu}$ and hence 
$\dbar m_{n-\nu}=\omega'_{n-\nu}-\omega_{n-\nu}$.
We thus have that \eqref{moment2} holds with
$\omega_{n-\nu}$ if and only it holds with $\omega'_{n-\nu}$.

\begin{remark}\label{buffel}  
The proofs above for part (i) of the theorems can be seen as  concrete
realizations  of abstract arguments.
There is a  long exact sequence
\begin{multline*}
0\to H^0_K(X,\Ok)\to H^0(X,\Ok)\to H^0(X\setminus K,\Ok)\to \\
\to H^1_K(X,\Ok)\to H^1(X,\Ok)\to H^1(X\setminus K,\Ok)\to \cdots.
\end{multline*}
Since $X$ is Stein, $H^k(X,\Ok)=0$ for $k\ge 1$.  Thus 
$H^0(X,\Ok)\to H^0(X\setminus K,\Ok)$ is surjective if $H^1_K(X,\Ok)=0$,
and in the same way, for $q\ge 1$, we have that
$H^q(X\setminus K,\Ok)=0$ if (and only if) $H^{q+1}_K(X,\Ok)=0$.
\end{remark}

We now consider   $X\setminus A$ where $X$ is Stein and
$A$ is an analytic subset of positive codimension.
For convenience  we first  consider  the technical part
concerning  local solutions  with compact support.

\begin{prop}\label{badanka3}
Let $X$ be an analytic set defined in a neighborhood of the
closed unit ball $\bar{\B}\subset\C^N$,
$A$ an analytic subset of $X$, and 
let $x\in A$, and let $a$ be a holomorphic tuple such that
$A=\{a=0\}$ in a neighborhood of $x$ and let $d=\dim A$.
Assume that $f$ is in $\A_{q+1}$ in a \nbh of $x$, $\dbar f=0$, and that
$f$ has support in $\{|a|<t\}$ for some small $t$.
(We may assume that  $f=0$ close to $A$.)

\smallskip

\noindent (i)  If $0\leq q \leq \nu-d-2$, then one can find, in a neighborhood $U$ of $x$,
a $(0,q)$-form  $u$ in $\A_{q}$  with support in $\{|a|<t\}$
such that  $\dbar u =f$ in $X\setminus A \cap U$.

\smallskip

\noindent (ii)  If $0\le q=\nu-d-1$, then one can find such a solution if and only if
\begin{equation}\label{eqmoment}
\int_{X}f\w h\w \omega_{n-\nu}=0
\end{equation}
for all smooth $\dbar$-closed $(0,d)$-forms  $h$  such that
$\supp  h \cap \{|a|\leq t\}$ is compact and contained in the set where
$\dbar f=0$.
\end{prop}

\begin{proof}
Let $\chi_a$ be a cutoff function in $\B$, which in a neighborhood of $x$
satisfies that $\chi_a=1$ in a neighborhood of the support of $f$ and
$\chi_a=0$ in a neighborhood of $\{|a|\geq t\}$. 
Close to $x$ we can choose coordinates
$z=(z',z'')=(z_1',\ldots,z'_{d},z_1'',\ldots,z_{N-d}'')$ centered at $x$ so
that
$A \subset \{|z''|\leq |z'|\}$.
Let 
$H^a$ be a holomorphic $(1,0)$-form, as in the proof of Theorem~\ref{main2},
and define 
\begin{equation*}
g^a=\chi_a(z)-\dbar\chi_a(z)\wedge \frac{\sigma_a}{\nabla_{\eta}\sigma_a},\,\,\,
\sigma_a=\frac{\overline{a(z)}\cdot H^a}{|a(z)|^2-a(\zeta)\cdot \overline{a(z)}}.
\end{equation*}
Then $g^a$ is a smooth weight for $\zeta$ on the support of $f$.
Since $f$ is supported close to
$A$ we can choose a function $\chi=\chi(\zeta')$, which is $1$
close to
$x$ and such that $f\chi$ has compact support. Let 
$g=\chi-\dbar\chi\wedge \sigma/\nabla_{\eta}\sigma$ be the weight from
Example 2 but built from $z'$ and $\zeta'$. Our Koppelman formula now gives
that
\begin{equation*}
u=\K f=\int (HR\wedge g^a \wedge g\wedge B)_{N,N-1}\wedge f
\end{equation*}
has the desired properties provided that the obstruction term
\begin{equation*}
\Pr f=\int (HR\wedge g^a\wedge g)_{N,N}\wedge f
\end{equation*}
vanishes. Since $g$ is built from $\zeta'$, $g$ has at most degree $d$ in
$d\bar{\zeta}$. Moreover,
$HR$ has at most degree $N-\nu$ in $d\bar{\zeta}$ and
$g^a$ has no degree in $d\bar{\zeta}$. Thus, if $q\le\nu-d-2$, then
$(HR\wedge g^a\wedge g)_{N,N}\wedge f$ cannot have degree $N$ in $d\bar{\zeta}$
and so $\Pr f=0$ in that case. This proves (i).
If $q=\nu-d-1$, then 
\begin{equation*}
\Pr f=\dbar\chi_a(z)\w \int HR_{N-\nu}\wedge g_{d}\wedge \sigma_a\w(\dbar\sigma_a)^q f. 
\end{equation*}
Now, $H^a$ depends holomorphically on $\zeta$ and $g_{d}$ is
$\dbar$-closed since it is the top degree term of a weight. Also, $g$ has
compact support in the $\zeta'$-direction, so $\mbox{supp}(g)\cap\{|a|\leq
t\}$
is compact and thus $\Pr f=0$ if \eqref{eqmoment} is fulfilled.
On the other hand, it is clear that the existence of a solution
with support in $\{|a|<t\}$ implies \eqref{eqmoment}.
\end{proof}


\begin{proof}[Proof of Theorem~\ref{main3}] 
Arguing as in the proof of Corollary~\ref{nyportencor} above, we can conclude
from Proposition~\ref{badanka3}: {\it Given a point $x$ there is a \nbh
$U$ such that if $\phi\in\A_q(U\cap X\setminus  A)$ is  $\dbar$-closed,
$0\le q\le\nu-d-2$ or  $0\le q=\nu-d-1$ and \eqref{moment3} holds, 
$\phi$ is strongly holomorphic if $q=0$ and exact in
$X\setminus A\cap U'$,  for a  possibly
slightly smaller \nbh $U'$ of $x$,  if $q\ge 1$.}

\smallskip
We define the analytic sheaves
$\F_k$ on $X$ by 
$\F_k(V)=\A_k(V\setminus A)$ for open sets $V\subset X$. Then
$\F_k$ are fine sheaves and 
\begin{equation}\label{karv}
0\to\Ok_X \to\F_0\stackrel{\dbar}{\longrightarrow}\F_1
\stackrel{\dbar}{\longrightarrow}\F_2\stackrel{\dbar}{\longrightarrow} \cdots
\end{equation}
is exact for $k\le\nu-d-2$.
It follows that 
$$
H^k(X,\Ok_X)=\frac{\Ker_{\dbar} \F_k(X)}
{\dbar \F_{k-1}(X)}
$$
for $k\le \nu-d-2$. Hence Theorem~\ref{main3} follows for $q\le\nu-d-2$. 
If $q=\nu-d-1$ and \eqref{moment3} holds, then $\phi$ is in the image of
$\F_{q-1}\to\F_q$, and then the result follows as well.
\end{proof}

\section{Examples}\label{exsec}

We have already seen that if $X$ is smooth, then
$\omega_k$ is just a smooth $(n,k)$-form, and
$\omega_0$ is non-vanishing.  At least semi-globally we can choose
$\omega=\omega_0$, and then $\omega_0$ is holomorphic.

Let now $X=\{h=0\}\subset \B\subset \C^{n+1}$,   $h\in \mathcal{O}(\bar{\B})$, 
be a hypersurface in the unit ball in $\C^{n+1}$ and assume that $0\in X$.
The depth (homological codimension) of $\Ok^X_x$ equals $\mbox{dim}\, X=n$ for all $x\in X$.
The residue current associated with $X$ is simply $R=R_1=\dbar (1/h)$ and so 
by the Poincare-Lelong formula \eqref{pl}
\begin{equation*}
R\wedge d\zeta = \dbar \frac{1}{h} \wedge d\zeta = 
\dbar \frac{1}{h}\wedge \frac{dh}{2\pi i}\wedge \tilde{\omega}=
\tilde{\omega}\wedge [X],
\end{equation*}
where, e.g., 
\begin{equation*}
\tilde{\omega}= 2\pi i \sum_{j=1}^{n+1}(-1)^{n-1}\frac{\overline{(\partial h/\partial \zeta_j)}}{|dh|^2}
d\zeta_1\wedge \cdots \wedge \widehat{d\zeta_j}\wedge \cdots \wedge d\zeta_{n+1}.
\end{equation*}
The structure  form associated with $X$ then is $\omega=i^* \tilde{\omega}$, 
where $i\colon X\hookrightarrow \B$.
Alternatively, we can write $R=\gamma \lrcorner [X]$, and thus
\begin{equation}\label{pro}
\omega=\pm i^*(\gamma\lrcorner d\zeta_1\w\ldots\w d\zeta_{n+1}),
\end{equation}
for 
\begin{equation}\label{qwer}
\gamma = -2\pi i\sum_{j=1}^{n+1} \frac{\overline{(\partial h/\partial \zeta_j)}}
{|dh|^2} \frac{\partial}{\partial \zeta_j}.
\end{equation}

Let $K=\{0\}\subset X$ and let $\phi\in \mathcal{A}_q(X\setminus K)$ be $\dbar$-closed. 
Since $\nu=n$ it follows from Theorem~\ref{nyporten} and Corollary~\ref{nyportencor}
that $\phi$ has a $\dbar$-closed extension in $\A_q(X)$ and is $\dbar$-exact in $X\setminus K$ if $q\leq n-2$,
or if $q=n-1$ and \eqref{moment2} holds. Let us consider \eqref{moment2} in our special case; 
assume therefore that $q=n-1$.
The function $\chi$ in \eqref{moment2} may be any 
smooth function that is $1$ in a neighborhood of $K$ and has compact support in $\B$. 
Via Stokes' theorem, or a simple limit procedure, we can write the condition \eqref{moment2} as
\begin{equation}\label{stu}
0=\int_{X\cap\partial\B_\epsilon} \omega\w\phi \xi,\quad  \xi\in \mathcal{O}(\B),
\end{equation}
where $\omega$ is given by \eqref{pro} and \eqref{qwer}.

In case $X=\{\zeta_{n+1}=0\}$ we have $\omega=\pm 2\pi i d\zeta_1\wedge \cdots \wedge d\zeta_n$ and 
\eqref{stu} reduces to the usual condition for $\phi$ having a $\dbar$-closed extension across $0$.
Let instead $X=\{\zeta_1^r-\zeta_2^s=0\}\cap \B\subset \C^2$, where $2\leq r < s$ are relatively prime integers. 
Then $\tau \mapsto (\tau^s,\tau^r)$ is the normalization of $X$. We have
\begin{equation*}
\gamma =-2\pi i\frac{r\bar{\zeta}_1^{r-1}\partial/\partial\zeta_1 -
s\bar{\zeta}_2^{s-1}\partial/\partial\zeta_2}{r^2|\zeta_1|^{2(r-1)}+s^2|\zeta_2|^{2(s-1)}},
\end{equation*}
and it is straightforward to verify that $\omega=2\pi i d\tau/\tau^{(r-1)(s-1)}$.
Let $\phi$ be holomorphic on $X\setminus\{0\}=X_{reg}$. 
Then, cf., \eqref{stu}, $\phi$ has a (strongly) holomorphic extension
to $X$ if and only if 
\begin{equation*}
\int_{|\tau|=\epsilon}\phi \xi \, d\tau/\tau^{(r-1)(s-1)}=0, \quad \xi \in \mathcal{O}(X).
\end{equation*}

\section{Proof of Theorem~\ref{mainglobal}}\label{mg}

We now turn our attention to the proof of Theorem~\ref{mainglobal}.
We first assume that $X$ is a subvariety of some domain $\Omega$ in $\C^N$.
A basic problem with the globalization is that we cannot assume that there
is one single resolution \eqref{acomplex} of $\Ok/\J$ in the whole
domain $\Omega$. We therefore must patch together local solutions. To this end we will
use Cech cohomology. Recall that if $\Omega_j$ is an open cover of $\Omega$, then
a $k$-cochain $\xi$ is a formal sum
$$
\xi=\sum_{|I|=k+1}\xi_I\w\epsilon_I
$$
where $I$ are multi-indices and $\epsilon_j$ is a nonsense basis, cf., e.g.,
\cite[Section~8]{A1}. Moreover, in this language
the coboundary operator $\rho$ is defined as
$\rho \xi = \epsilon\w \xi$,
where $\epsilon=\sum_j \epsilon_j$.

If $g$ is a weight as in Example~\ref{alba}  and $g'=(1-\chi)\sigma/\nabla_\eta \sigma$, then
\begin{equation}\label{gprim}
\nabla_\eta g'=1-g.
\end{equation}

Notice that the relations \eqref{Hdef} for the  Hefer morphism(s) 
can be written simply as
$$
\delta_\eta H=Hf-f(z)H=Hf
$$
if $z\in X$.

\begin{proof}[Proof of Theorem~\ref{mainglobal} in case $X\subset\Omega\subset \C^N$] 
Assume that $\phi$ is in $\W(X)\cap\Dom\dbar_X$, smooth on $X_{reg}$, and that
$\dbar\phi=0$.
Let $\Omega_j$ be a locally finite  open cover of $\Omega$  with convex polydomains
(Cartesian products of convex domains in each variable), and for each $j$ let 
$g_j$ be   a  weight with support in a slightly larger convex
polydomain  $\tilde\Omega_j\supset\supset\Omega_j$
and holomorphic in $z$ in a \nbh  of $\overline\Omega_j$.
Moreover, for each $j$ suppose that we have a given resolution \eqref{acomplex}
in $\tilde\Omega_j$, 
a choice of Hermitian metric, a choice of Hefer morphism, and let  $(HR)_j$ be   the
resulting current.  Then, cf., Remark~\ref{zoran} above, 
\begin{equation}\label{locsol}
u_j(z)=\int \big((HR)_j\w g_j\w B\big)_{N,N-1}\w\phi
\end{equation}
is a solution in $\Omega_j\cap X_{reg}$ to $\dbar u_j=\phi$.
We will prove that $u_j-u_k$ is (strongly) holomorphic on $\Omega_{jk}\cap X$
if $q=1$ and $u_j-u_k=\dbar u_{jk}$ on $\Omega_{jk}\cap X_{reg}$ if
$q>1$, and more generally:

\smallskip
\noindent {\bf Claim I\ } {\it Let $u^0$ be the $0$-cochain $u^0=\sum u_j\w\epsilon_j$. For each
$k\le q-1$ there is  a $k$-cochain of $(0,q-k-1)$-forms on $X_{reg}$ such that
$\rho u^k=\dbar u^{k+1}$ if $k<q-1$ and  $\rho u^{q-1}$ is a 
(strongly) holomorphic $q$-cocycle.} 
\smallskip

The holomorphic $q$-cocycle $\rho u^{q-1}$ defines a class in $H^q(\Omega,\Ok/\J)$ and
if $\Omega$ is pseudoconvex this class  must vanish,
i.e., there is a holomorphic $q-1$-cochain $h$ such that $\rho h=\rho u^{q-1}$.
By standard arguments this yields a global solution to $\dbar\psi=\phi$. For instance,
if $q=1$ this means that we have holomorphic functions $h_j$ in $\Omega_j$ such that
$u_j-u_k=h_j-h_k$ in $\Omega_{jk}\cap X$. It follows that $u_j-h_j$ is a global
solution in $X_{reg}$. 

\smallskip

We thus have to prove Claim~I.
To begin with we assume that we have a fixed resolution with a fixed metric 
and Hefer morphism;  thus 
a fixed choice  of current $HR$.  Notice that if
$$
g_{jk}=g_j\w g_k'-g_k\w g_j',
$$
cf., \eqref{gprim}, then 
$$
\nabla_\eta g_{jk}=g_j-g_k
$$
in $\tilde\Omega_{jk}$. With $g^\lambda$ as in Section~\ref{koppsec}, 
and in view of \eqref{bm},  we have 
$$
\nabla_\eta(g^\lambda\w g_{jk}\w  B)=
g^\lambda\w g_j\w B-g^\lambda\w g_k\w B - g^\lambda\w  g_{jk} 
+ g^\lambda\w g_{jk}\w [\Delta].
$$
However, the last term must vanish since $[\Delta]$ has full degree in
$d\eta$ and $g_{jk}$ has at least degree $1$. Therefore
$$
-\dbar(g^\lambda\w g_{jk}\w  B)_{N,N-2}= (g^\lambda\w g_j\w B)_{N,N-1}-
(g^\lambda\w g_k\w B)_{N,N-1} -(g^\lambda\w g_{jk})_{N,N-1}
$$
and as in  Section~\ref{koppsec}  we can take $\lambda=0$ and get, 
assuming  that $\dbar\phi=0$  and arguing as in Remark~\ref{zoran}, 
\begin{equation}\label{skillnad}
u_j-u_k=
\int (HR\w g_{jk})_{N,N-1}\w \phi +\dbar_z\int (HR\w g_{jk}\w B)_{N,N-2}\w \phi.
\end{equation}
Since $g_{jk}$ is holomorphic in $z$ in $\Omega_{jk}$ it follows that
$u_j-u_k$ is  (strongly)  holomorphic in $\Omega_{jk}\cap X$ if $q=1$ and
$\dbar$-exact on $\Omega_{jk}\cap X_{reg}$  if $q>1$.


\smallskip
\noindent {\bf Claim II\ }
{\it Assume that we have a fixed resolution but different choices of
Hefer forms and metrics and thus different
$a_j=(HR)_j$ in $\tilde\Omega_j$. Let $\epsilon'_j$ be a nonsense basis.
If
$A^0=\sum a_j\w\epsilon'_j$, then for each $k>0$ there is a $k$-cochain 
$$
A^k=\sum_{|I|=k+1}A_I\w\epsilon_I', 
$$
where  $A_I$ are currents on $\tilde\Omega_I$ with support
on $\tilde\Omega_I\cap X$ and holomorphic in $z$ in $\Omega_I$, such that
\begin{equation}\label{bus3}
\rho'A^k=\epsilon'\w A^k=\nabla_\eta A^{k+1}.
\end{equation}
Moreover, 
\begin{equation}\label{zoran2}
\dbar\chi_\delta \w\phi\w A^k\to 0,\quad \delta\to 0.
\end{equation}}
\smallskip

For the last statement we use that $X$ is Cohen-Macaulay.
\smallskip

In particular we have currents 
$a_{jk}$  with support on $X$ and such that
$
\nabla_\eta a_{jk}=a_j-a_k
$
in $\tilde\Omega_{jk}$.
If
$$
w_{jk}=a_{jk}\w g_j\w g_k+a_j\w g_j\w g_k'-a_k\w g_k\w g_j',
$$
then
$$
\nabla_\eta w_{jk}=a_j\w g_j-a_k\w g_k.
$$
Notice that $w_{jk}$ is a globally defined current.
By a similar argument as above (and via a suitable limit process),
cf., Remark~\ref{zoran} and \eqref{zoran2},  one gets that
$$
u_j-u_k=
\int(w_{jk})_{N,N-1}\w\phi+\dbar_z\int(w_{jk}\w B)_{N,N-2}\w \phi
$$
in $\Omega_{jk}\cap X_{reg}$ as before.
In general  we put 
$$
\epsilon'=g=\sum g_j\w\epsilon_j.
$$ 
If, cf.,  \eqref{gprim},
$$
g'=\sum g'_j\w\epsilon_j
$$
then
$$
\nabla_\eta g'=\epsilon-g=\epsilon-\epsilon'.
$$
If $a_I$ is a form on $\tilde\Omega_I$, then
$a_I\w\epsilon'_I$ is a well-defined global form. 
Therefore $A$, and hence also 
$$
W=A\w e^{g'},
$$ 
i.e., $W^k=\sum_j A^{k-j} (g')^{j}/j!$,
has globally defined coefficients and
$$
\rho W=\nabla_\eta W.
$$
In fact, since $A$ and $g'$ have even degree, 
$$
\nabla_\eta(A\w e^{g'})=\epsilon'\w A\w e^{g'}+
A\w e^{g'}\w (\epsilon-\epsilon')=\epsilon\w A\w e^{g'}.
$$
By the yoga above  the $k$-cochain 
$$
u^k=\int (W^k\w B)_{N,N-k-1}\w\phi
$$
satisfies
$$
\rho u^k=\dbar_z \int (W^{k+1}\w B)_{N,N-k-2}\phi +\int(W^{k+1})_{N,N-k-1}\w\phi.
$$
Thus  $\rho u^k=\dbar u^{k+1}$ for $k<q-1$ whereas $\rho\w u^{q-1}$
is a holomorphic $q$-cocycle as desired.

\smallskip
It remains to consider the case when we have different resolutions in $\Omega_j$.
For each pair $j,k$
choose a weight $g_{s_{jk}}$ with support in $\tilde\Omega_{jk}$ that is 
holomorphic in $z$ in $\Omega_{s_{jk}}=\Omega_{jk}$.
By \cite[Theorem~3 Ch.~6 Section~F]{GR}
we can choose a resolution in $\tilde\Omega_{s_{jk}}=\tilde\Omega_{jk}$ 
in which  both of the
resolutions in $\tilde\Omega_j$ and $\tilde\Omega_k$ restricted to $\Omega_{s_{jk}}$ 
are direct summands.
Let us fix metric and Hefer form and thus a current $a_{s_{jk}}=(HR)_{s_{jk}}$ 
in $\Omega_{s_{jk}}$
and thus a  solution $u_{s_{jk}}$ corresponding to $(HR)_{s_{jk}}\w g_{s_{jk}}$.
If we extend the metric and Hefer form from $\tilde\Omega_j$ in a way that respects
the direct sum, then $(HR)_j$ with these extended choices will be unaffected,
cf., \cite[Section~4]{AW1}. On $\tilde\Omega_{js_{jk}}$ we
therefore  practically speaking 
have  just one single resolution and as before thus  $u_j-u_s$
is holomorphic (if $q=1$) and $\dbar  u_{j s_{jk}}$ if $q>1$.
It follows that 
$u_j-u_k=u_j-u_s+u_s-u_k$ is holomorphic on $\Omega_{jk}$ if $q=1$ and equal to
$\dbar$ of 
$$
u_{jk}=u_{js_{jk}}+u_{s_{jk}k}
$$
if $q>1$.
We now claim that each $1$-cocycle 
\begin{equation}\label{bart}
u_{jk}+u_{kl}+u_{lj}
\end{equation}
is holomorphic on $\Omega_{jkl}$ if $q=2$ and
$\dbar$-exact on $\Omega_{jkl}\cap X_{reg}$ if $q>2$.
On $\tilde\Omega_{s_{jkl}}=\tilde\Omega_{jkl}$ we can choose a resolution in which
each of the resolutions associated with the indices $s_{jk}, s_{kl}$ and $s_{kj}$
are direct summands. It follows that 
$
u_{js_{jk}}+u_{s_{jk}s_{jkl}}+u_{s_{jkl}j}
$
is holomorphic if $q=2$ and $\dbar u_{j s_{jk}s_{jkl}}$ if $q>2$.
Summing up, the statement about \eqref{bart} follows.
If we continue in this way Claim~I follows.

\smallskip
It remains to prove Claim~II.
It is not too hard to check by an appropriate
induction procedure, cf., the very construction
of Hefer morphisms in \cite{A7}, that if we have two choices
of (systems of) Hefer forms $H_j$ and $H_k$ for the same
resolution $f$, then there is a form
$H_{jk}$ such that 
\begin{equation}\label{tarta1}
\delta_\eta H_{jk}=H_j-H_k+f(z)H_{jk}-H_{jk}f.
\end{equation}
More generally, if
$$
H^0=\sum H_j\w\epsilon_j
$$
then for each $k$ there is a  (holomorphic) $k$-cochain $H^k$ such that
(assuming $f(z)=0$ for simplicity)
\begin{equation}\label{bus1}
\delta_\eta H^k=\epsilon\w H^{k-1}-H^k f
\end{equation}
(the difference in sign between \eqref{tarta1} and \eqref{bus1} is because
in the latter one $f$ is to the right of the basis elements).

Elaborating the construction in \cite[Section~4]{AW1}, cf., 
\cite[Section~8]{A1},  one  finds, given $R^0 
=\sum R_j\w\epsilon_j$,  $k$-cochains of currents $R^k$ such that 
\begin{equation}\label{bus2}
\nabla_f R^{k+1}= \epsilon\w  R^k.
\end{equation}
(With the notation in \cite{AW1}, if $R_j=\dbar|F|^{2\lambda}\w u^j|_{\lambda=0}$,
then the coefficient for $\epsilon_j\w \epsilon_k\w \epsilon_\ell$ is 
$\dbar|F|^{2\lambda}\w u^ju^ku^\ell|_{\lambda=0}$, etc.)


We  define a product of forms in the following way.
If  the multiindices $I,J$ have no index in common, then
$(\epsilon_I,\epsilon_J)=0$, whereas
$$
(\epsilon_I\w\epsilon_\ell,\epsilon_\ell\w\epsilon_J)=\frac{|I|!|J|!}{(|I|+|J|+1)!}
\epsilon_{I}\w\epsilon_J.
$$
We then extend it to any forms bilinearly in the natural way.
It is easy to check that 
$$
(H^k f, R^\ell)=-(H^k, f R^\ell).
$$
Using  \eqref{bus1} and \eqref{bus2} (and keeping in mind that
$H^k$ and $R^\ell$ have odd order) one can verify that 
$$
\nabla_\eta(H^k,R^\ell)=(\epsilon\w H^{k-1}, R^\ell)
+(H^k, \epsilon\w R^\ell).
$$
By a similar argument one can  finally check  that
$$
A^k=\sum_{j=0}^k(H^j,R^{k-j})
$$
will satisfy \eqref{bus3}.

\smallskip
Since  $X$ is Cohen-Macaulay, each $R^k$ will be 
a smooth form times the principal term $(R_j)_{N-n}$ for $R_j$
corresponding to some choice of metric. The case with
two different metrics is described in \cite[Section~4]{AW1} and
the general case is similar;   compare also to the discussion preceding
Remark~\ref{buffel}.
Thus \eqref{zoran2} holds, and
thus  Claim~II holds, and so   Theorem~\ref{mainglobal}  is proved in case 
$X$ is a subvariety of $\Omega\subset\C^N$. 
\end{proof}

\begin{remark}\label{zoran3}
If $X$ is not Cohen-Macaulay, then we must assume explicitly that
$\dbar\chi_\delta\w \phi \w R^k\to 0$ for all $R^k$. 
\end{remark}

The extension to a general analytic space $X$ is done in pretty much the same
way and we just sketch the idea. First assume that we have
a fixed $\eta$ as before but two different choices $s$ and $\tilde s$ of
admissible form, and let $B$ and $\tilde B$ be the corresponding locally
integrable forms. Then, one can check, arguing as
in \cite[Section~5]{AS2}, that
\begin{equation}\label{bbprim}
\nabla_\eta (B\w\tilde B)=\tilde B-B
\end{equation}
in the current sense, 
and by a minor modification of Lemma~\ref{brum}
one can  check that 
$$
\int (HR\w g\w B\w \tilde B)_{N,N-2}\w\phi
$$
is smooth on $X_{reg}\cap\Omega'$; for degree reasons it  vanishes if $q=1$.
It follows from \eqref{bbprim} that 
$
\nabla_\eta(HR^\lambda\w g\w B\w \tilde B)=
HR^\lambda\w g\w \tilde B-
HR^\lambda\w g\w B
$
from which we can conclude that
\begin{multline}\label{hast}
\dbar_z\int (HR\w g\w B\w\tilde B)_{N,N-2}\w\phi=\\
\int(HR\w g\w B)_{N,N-1}\w\phi-\int(HR\w g\w \tilde B)_{N,N-1}\w\phi,
\quad  z\in \Omega'\cap X_{reg}.
\end{multline}

Now  let us assume  that
we have two local solutions,  in say $\Omega$ and $\Omega'$,  obtained from
two different embeddings of slightly larger sets $\tilde \Omega$ and $\tilde\Omega'$
in subsets of $\C^{N}$ and $\C^{N'}$, respectively. We want to compare these
solutions on $\Omega\cap\Omega'$. 
Localizing further, as before, we may assume that the weights 
both have support in $\tilde \Omega\cap\tilde \Omega'$.  After adding nonsense
variables  we may assume that both embeddings are into the same
$\C^N$, and after further localization there is a local
biholomorphism in $\C^N$ that maps one embedding onto the other one,
see \cite{GR}.  
(Notice that a solution  obtained via an embedding in $\C^{N_1}$
also can be obtained via an embedding into a larger $\C^N$, by just adding
dummy variables in the first formula.)
In other words, we may assume that we have the same
embedding in some open set $\Omega\subset\C^N$ but two solutions obtained  from
different  $\eta$ and $\eta'$. 
(Arguing as before, however, we may assume that we have the same resolution
and the same residue current $R$.)
Locally there is an  invertible matrix $h_{jk}$ such that
\begin{equation}\label{arvika}
\eta'_j=\sum h_{jk}\eta_k.
\end{equation}
We define a vector bundle mapping
$
\alpha^*\colon \Lambda_{\eta'}\to\Lambda_\eta
$
as  the identity on $T^*_{0,*}(\Omega\times\Omega)$
and so that  
$$
\alpha^* d\eta_j'=\sum h_{jk} d\eta_k.
$$
It is readily checked that
$$
\nabla_\eta\alpha^*= \alpha^*\nabla_{\eta'}.
$$
Therefore, $\alpha^*g'$ is an $\eta$-weight if $g'$ is an $\eta'$-weight.
Moreover,  if $H$ is an $\eta'$-Hefer morphism, then $\alpha^*H$
is an $\eta$-Hefer morphism, cf., \eqref{Hdef}. 
If $B'$ is obtained from an $\eta'$ admissible form $s'$, then
$\alpha^*s'$ is an $\eta$-admissible form and
$\alpha^*B'$ is the corresponding locally integrable form.
We  claim  that the $\eta'$-solution
\begin{equation}\label{prall}
v'=\int (H'R\w g'\w B')_{N,N-1}\w\phi
\end{equation}
is comparable to the $\eta$-solution
\begin{equation}\label{prall2}
v=\int \alpha^*(H'R)\w\alpha^* g'\w \alpha^* B'\w\phi.
\end{equation}
Notice that we are only interested in the $d\zeta$-component of the
kernels.
We have that ($d\eta=d\eta_1\w\ldots\w d\eta_N$ etc)
$$
(H'R\w g'\w B')_{N,N-1}=A\w d\eta' \sim
A\w\det(\partial\eta'/\partial\zeta) d\zeta
$$
and 
$$
\alpha^*(H'R\w g'\w B')_{N,N-1}=A\w\det h\w d\eta
\sim A\w\det h \det (\partial\eta/\partial\zeta) d\zeta.
$$ 
Thus
$$
\alpha^*(H'R\w g'\w B')_{N,N-1}\sim \gamma(\zeta,z)(H'R\w g'\w B')_{N,N-1}
$$
with
$$
\gamma=\det h \det\frac{\partial\eta}{\partial\zeta}
\Big(\det\frac{\partial\eta'}{\partial\zeta}\Big)^{-1}.
$$
From  \eqref{arvika} we have that
$\partial\eta'_j/\partial\zeta_\ell=
\sum_k h_{jk}\partial\eta_k/\partial\zeta_{\ell} +\Ok(|\eta|)
$
which implies that $\gamma$ is $1$ on the diagonal.
Thus  $\gamma$ is a smooth (holomorphic) weight and therefore 
\eqref{prall} and \eqref{prall2} are comparable, and thus the claim
is proved. This proves Theorem~\ref{mainglobal} in the case $q=1$, and
elaborating the idea  as in the
previous proof we obtain the general case.

\begin{remark}
In case $X$ is a Stein space and $X_{sing}$ is discrete there is a much simpler proof
of Theorem~\ref{mainglobal}. 
To begin with we can solve $\dbar v=\phi$ locally, and 
modifying by such local solutions we may assume that
$\phi$ is vanishing identically in a \nbh of $X_{sing}$.
There exists  a sequence of holomorphically convex open subsets
$X_j$ such that $X_j$ is relatively compact in $X_{j+1}$ and
$X_j$ can be embedded as a subvariety of some pseudoconvex set
$\Omega_j$ in $\C^{N_j}$. Let $K_\ell$ be the closure of $X_\ell$.
By Theorem~\ref{main} we can solve $\dbar u_\ell =\phi$ in a \nbh of 
$K_\ell$ and  $u_\ell$ will be smooth. 
If $q>1$ we can thus  solve $\dbar w_\ell=u_{\ell+1}-u_\ell$
in a \nbh of $K_\ell$, and since $X_{sing}$ is discrete we can 
assume that $\dbar w_\ell$ is smooth in $X$.  Then
$v_\ell=u_\ell-\sum_1^{\ell-1}\dbar w_k$ defines a global solution.
If $q=1$, then  one  obtains a global solution in a similar way
by  a Mittag-Leffler type argument. 
\end{remark}

\def\listing#1#2#3{{\sc #1}:\ {\it #2},\ #3.}

\end{document}